\documentclass[onefignum,onetabnum,final]{siamart220329}

\usepackage{lipsum} 
\usepackage{amsfonts}
\usepackage{graphicx}
\usepackage{color}
\usepackage{subfigure}
\usepackage{epstopdf}
\usepackage{algorithmic}
\usepackage{amssymb}
\usepackage[margin=1.35in]{geometry}
\ifpdf
  \DeclareGraphicsExtensions{.eps,.pdf,.png,.jpg}
\else
  \DeclareGraphicsExtensions{.eps}
\fi

\newsiamremark{remark}{Remark}
\newsiamremark{hypothesis}{Hypothesis}
\crefname{hypothesis}{Hypothesis}{Hypotheses}
\newsiamthm{claim}{Claim}

\headers{Scaling Optimized Hermite Approximation Methods}{H. Hu and H. Yu}

\title{Scaling Optimized Hermite Approximation Methods\thanks{Prepared on \today.
\funding{This work was partially supported by the National Natural Science
	Foundation of China under Grant No. 12171467, 12494543, 12161141017, Strategic Priority Research Program of Chinese Academy of Sciences under Grant XDA0480504 and National Key
	R\&D Program of China under Grant No. 2021YFA1003601.}
  }
  }

\author{Hao Hu
\thanks{School of Mathematical Sciences, University of Chinese Academy of Sciences, Beijing 100049, China; LSEC, Institute of Computational Mathematics and Scientific/Engineering Computing, Academy of Mathematics
and Systems Science, Beijing 100190, China
  (\email{huhao@lsec.cc.ac.cn}).}
  \and HaiJun Yu
  \thanks{Corresponding author. State Key Laboratory of Mathematical Sciences \& LSEC, Institute of Computational Mathematics and Scientific/Engineering
Computing, Academy of Mathematics and Systems Science, Beijing 100190, China; School of Mathematical Sciences,
University of Chinese Academy of Sciences, Beijing 100049, China
  (\email{hyu@lsec.cc.ac.cn}).}
}

\usepackage{amsopn}

\ifpdf
\hypersetup{
  pdftitle={Scaling Optimized Hermite Approximation Methods},
  pdfauthor={Hao Hu and Haijun Yu}
}
\fi

\begin{document}

\maketitle

\begin{abstract}
Hermite polynomials and functions have extensive applications in scientific and engineering problems. 
Although it is recognized that employing the scaled Hermite functions
rather than the standard ones can remarkably enhance the approximation
performance, the understanding of the scaling factor remains insufficient.
Due to the lack of theoretical analysis, recent publications still cast doubt on whether
the Hermite spectral method is inferior to other methods. To dispel this doubt,
we show in this article that the inefficiency of the Hermite spectral method
comes from the imbalance in the decay speed of the objective function within the spatial and frequency domains. 
Proper scaling can render the Hermite spectral methods comparable to other methods.
To make it solid, we propose a novel error analysis framework for the scaled Hermite approximation. Taking the $L^2$
projection error as an example, our framework
illustrates that there are three different components
of errors: the spatial truncation error, the frequency truncation error, and the Hermite spectral approximation error.
Through this perspective, finding the optimal scaling factor is equivalent
to balancing the spatial and frequency truncation errors. As applications, 
we show that geometric convergence can be recovered by proper scaling for a class of functions.
Furthermore, we show that proper scaling can double the convergence order for smooth functions with algebraic decay.
The perplexing pre-asymptotic sub-geometric convergence when approximating algebraic
decay functions can be perfectly explained by this framework.
\end{abstract}



\section{Introduction}
Spectral methods, which utilize orthogonal polynomials or functions as bases, constitute one of the most prevalent approaches to numerically solving partial
differential equations(PDEs). These methods possess a high level of approximation accuracy and hold a significant position in scientific and engineering computations. 

For unbounded regions, Hermite polynomials or Hermite functions are extremely useful, because they are often the exact unperturbed eigenfunctions~\cite{boyd_asymptotic_1984}. Their applicability spans a broad spectrum of fields, including Schr\"odinger equations on unbounded domains~\cite{bao_fourth-order_2005,shen_error_2013, sheng_nontensorial_2021}, Vlasov and kinetic equations~\cite{shan_discretization_1998, mieussens_discrete-velocity_2000, gibelli_spectral_2006, cai_numerical_2010,Mizerova.She2018,funaro_stability_2021,zhang_error_2023}, 
high-dimensional PDEs and stochastic
differential equations~\cite{luo_hermite_2013, zhang_sparse-grid_2013},
fluid dynamics and uncertainty quantification~\cite{cameron_orthogonal_1947,
meecham_wienerhermite_1964, orszag_dynamical_1967, xiu_modeling_2003, BabusNT2007StochasticCollocation,
venturi_wick-malliavin_2013, tang_discrete_2014, nobile_adaptive_2016,
wan_numerical_2019},
and even electronic design~\cite{manfredi_perturbative_2020, chen_analytical_2021}, to name a few. The orthogonality property intrinsically renders them the prime and natural choices for serving as bases in the numerical treatment of PDEs within unbounded domains. This unique characteristic subsequently gives rise to linear algebraic systems that exhibit favorable and well-conditioned behavior for linear problems with constant coefficients. Moreover, the fast
Gauss transform~\cite{greengard_fast_1991,wan_sharp_2006} can be effectively applied to construct highly efficient Hermite spectral methods for nonlinear and variable-coefficient problems.

However, Gottlieb and Orszag, in their seminal book on the spectral method~\cite{gottlieb_numerical_1977}, pointed out a significant drawback. They noted that, in contrast to Chebyshev or Legendre spectral methods within bounded domains, the standard Laguerre and Hermite spectral methods exhibit poor resolution capabilities. This observation is later supported by some rigorous error estimates for Hermite approximations~\cite{boyd_rate_1980,gautschi_error_1983, boyd_asymptotic_1984, guo_error_1999,  kazashi_suboptimality_2023,wang_convergence_2023}.
Although Gottlieb and
Orszag~\cite{gottlieb_numerical_1977} further indicated that the resolution of Laguerre and Hermite spectral methods can be enhanced through appropriate scaling; they are still suspected to be inferior to mapped spectral methods or bounded-domain spectral methods with truncation. 
To dispel this doubt, Tang~\cite{tang_hermite_1993} proposed a straightforward but remarkably effective approach in 1993 for the selection of the scaling factor pertinent to Gaussian-type functions. This innovation led to a substantial elevation in the approximation efficiency of Hermite spectral
methods. Subsequently, the concept of scaled spectral expansion has gained wide acceptance. Numerical investigations show that, in certain scenarios, scaled Hermite/Laguerre
methods surpass mapped spectral methods (see, e.g.,~\cite{shen_recent_2009,shen_spectral_2011,huang_improved_2024}). 
Meanwhile, practical algorithms that incorporate scaling techniques have also been developed.
For example, Ma et al.~\cite{ma_hermite_2005} proposed a time-dependent scaling strategy designed for parabolic equations, while frequency-dependent scaling mechanisms have also been developed~\cite{xia_efficient_2021, chou2023adaptive}.

Despite the progress made, dedicated and comprehensive error analyses that focus specifically on the scaling factor remain scarce.
In fact, even for the standard Hermite approximation, important gaps remain in our understanding.
The root-exponential rate $\exp(-C\sqrt{N})$ has been reported for analytic functions without explicit proof~\cite{barrett1961convergence,davis2007methods}.
Shen et al.~\cite{shen_new_2000,shen_recent_2009,shen_spectral_2011} also reported this rate when approximating algebraic decay functions in the pre-asymptotic range. This phenomenon vividly illustrates the complexity of the Hermite approximation and
is considered puzzling, since the error estimate only predicts a convergence order of about $N^{-h}$, here $h$ is related to the algebraic decay rate.
Boyd did some pioneering works, asymptotic results of Hermite coefficients for certain specific entire functions and analytic functions with poles were obtained by using the method of steepest descent~\cite{boyd_rate_1980,boyd_asymptotic_1984},
while the exponential convergence of Hermite approximation remains unexplained.
We notice that until recently, as we know, the sub-geometric convergence $\exp(-C\sqrt{N})$
for some analytic functions has finally been proven by Wang et al.~\cite{wang_convergence_2023, wang_convergence_2024}.
However, their result cannot explain the puzzling pre-asymptotic behavior
mentioned previously, and the exponential convergence rate
$\exp(-CN^\alpha)$ with $\alpha \ne \frac{1}{2}$ also remain unexplained.

Returning to the issue of the scaling factor, due to the lack of a solid theoretical analysis, some recent theoretical results on ``standard''(without scaling optimization) Hermite method lead to the conclusion that the ``standard'' Hermite method is inferior to other spectral methods. For example, 
Kazashi et al.~\cite{kazashi_suboptimality_2023} compare the Gauss--Hermite quadrature with the trapezoidal rule. They proved that standard Gauss--Hermite quadrature is only ``sub-optimal'', since for functions with $\alpha$ order smoothness in some sense, only an order about $N^{-\alpha/2}$ with $N$ function evaluations by Gauss--Hermite quadrature 
can be achieved.
In contrast, a suitably truncated trapezoidal rule achieves about $N^{-\alpha}$ up to a logarithmic factor. Kazashi et al. also mentioned that Sugihara~\cite{sugihara1997optimality} established the rate $\exp(-C N^{\rho / (\rho+1)})$ by the trapezoidal rule for functions decaying at the rate $\exp(-\tilde{C}|x|^\rho)\,\,(\rho \geq 1)$
on the real axis under other suitable assumptions. However, the existing literature
only shows that the Hermite approximation can only achieve a rate $\exp(-C\sqrt{N})$.

Trefethen~\cite{trefethen2022exactness} also pointed out that for functions analytic in
a strip with $\exp(-x^2)$ decay on the real axis, the Gauss--Legendre, Clenshaw--Curtis, and trapezoidal
quadrature can achieve a convergence rate of $\exp(-CN^{2/3})$. However, numerical results show
that the Gauss--Hermite quadrature can only achieve the $\exp(-C\sqrt{N})$ rate. Although Trefethen mentioned that
Weideman showed that $\exp(-CN^{2/3})$ can be achieved by Gauss--Hermite quadrature with a proper scaling in an unpublished work at a conference in 2018, we have not seen any subsequent published work that strictly proves this statement.

Such issues, collectively compelling a sense of urgency and confusion, prompted Professor Trefethen to offer a pessimistic assessment in hist recent paper published on SIAM Review~\cite{trefethen2022exactness}: ``The literature seems not to confront the conceptual question: What has gone wrong with the Gauss--Hermite notion of optimality? ''
This fundamental critique underscores a critical gap in our understanding, highlighting the profound significance and inherent perplexity surrounding the very foundations of this widely used approximation method. Resolving this impasse is not merely an academic exercise; it is essential for advancing the reliability and applicability of spectral methods based on Hermite polynomials.

In this paper, we provide a rigorous and impactful answer to this confusing impasse: the imbalance of decay speeds in the spatial and frequency domains leads to an inefficient utilization of collocation points, by devising a novel framework for the error analysis of scaled
Hermite spectral methods. This frame serves as a guide for the selection of the optimal scaling factor. All the aforementioned issues can be effectively and comprehensively resolved within this framework.
By using a proper scaling factor derived from our framework, one can
easily characterize the unusual convergence rate of the Hermite approximation. For example, for functions smooth enough with $\exp(-\tilde{C}|x|^\rho), \rho\ge 1$ decay (including the functions studied by Trefethen~\cite{trefethen2022exactness} and Sugihara~\cite{sugihara1997optimality}), 
the scaling optimized Hermite approximations achieve convergence rate $\exp(-CN^{\rho/(\rho+1)})$ (see \cref{sec: Hermite optimality} and \cref{subsec: exp conv}). 
For the $\alpha$-order smooth functions studied by Kazashi et al.~\cite{kazashi_suboptimality_2023}, the scaled Gauss--Hermite quadrature based on our framework can also achieve a convergence rate $N^{-\alpha}$ up to a logarithmic factor (see \cref{sec: Hermite optimality}).
Moreover, the puzzling pre-asymptotic behavior can also be lucidly explicated (see \cref{subsec: puzzling convergence}).

\medskip
Next, we give a brief mathematical introduction to Hermite functions and our main idea.

\subsection{Basics on Hermite functions}

The Hermite polynomials, defined on the entire line $\mathbb{R}:=(-\infty,+\infty)$,
are orthogonal with respect to the weight function $\omega(x) = e^{-x^2}$, namely,
\begin{equation}\label{eq:HermPoly-def}
  \int_{-\infty}^{+\infty} H_m(x) H_n(x) \omega(x) d x=\gamma_n \delta_{m n}, \quad \gamma_n=\sqrt{\pi} 2^n n!.
\end{equation}
Hermite functions are defined by
\begin{equation}\label{eq:HermFunc-def}
  \widehat{H}_n(x)=\frac{1}{\pi^{1 / 4} \sqrt{2^n n!}} e^{-x^2 / 2} H_n(x), \quad n \geq 0,\, x \in \mathbb{R},
\end{equation}
which are normalized so that
\begin{equation}\label{eq:Herm-orth}
  \int_{-\infty}^{+\infty} \widehat{H}_n(x) \widehat{H}_m(x) d x=\delta_{m n}.
\end{equation}
Hermite functions satisfy the three-term recurrence relation
\begin{equation}\label{eq:three-recur}
  \widehat{H}_{n+1}(x)=x \sqrt{\frac{2}{n+1}} \widehat{H}_n(x)-\sqrt{\frac{n}{n+1}} \widehat{H}_{n-1}(x), \quad n \geq 1, 
\end{equation}
with $\widehat{H}_0= \frac{1}{\pi^{1/4} } e^{-x^2/2}$,
$\widehat{H}_1= \frac{\sqrt{2}}{\pi^{1/4}} e^{-x^2/2}x $.
The derivatives of the Hermite functions satisfy
\begin{equation}\label{eq:Herm-derivative}
  \begin{aligned}
  \widehat{H}_n^{\prime}(x) & =\sqrt{2 n} \widehat{H}_{n-1}(x)-x \widehat{H}_n(x) \\
  & =\sqrt{\frac{n}{2}} \widehat{H}_{n-1}(x)-\sqrt{\frac{n+1}{2}} \widehat{H}_{n+1}(x), \quad n\ge 1.
  \end{aligned}
\end{equation}
More details on the properties of the Hermite functions can be found in Section
7.2.1 of~\cite{shen_spectral_2011}, or Chapter 18 of~\cite{olver_nist_2010}.

\subsection{An intuitive explanation of our idea}
\label{subsec:intuitive}
Let $P_N$ denote the collection of all polynomials with degree no more than $N$, define
$\widehat{P}_N$ as 
\begin{equation}
  \widehat{P}_N:=\bigl\{\phi: \phi=e^{-x^2 / 2} \psi,\, \forall\,\psi \in P_N\bigr\} .
\end{equation}
In Hermite interpolation, we use the roots of $\widehat{H}_{N+1}$, 
which are denoted by $\{x_j\}_{j=0}^N$ with order $x_0 < x_1 < \cdots < x_N$,
as the collocation points to reconstruct a function in $\widehat{P}_N$.
It is known (see, e.g., equation (7.86) in~\cite{shen_spectral_2011})
\begin{equation}
  \max _j\left|x_j\right| \sim \sqrt{2 N}.
\end{equation}
If we call the interval $\left[x_0, x_N\right]$ the collocation interval, the information outside the collocation interval is not used in interpolation; 
thus, if the interpolated function outside $\left[x_0, x_N\right]$, roughly $\bigl[-\sqrt{2N},\sqrt{2N}\bigr]$,
is not negligible (e.g. having a decay rate slower than Hermite functions), one cannot expect a good approximation. When using scaled
Hermite functions $\widehat{H}_n(\beta x)$, the exterior error outside
$\left[-\sqrt{2N}/\beta, \sqrt{2N}/\beta\right]$ should be considered.
In other words, we expect the spatial truncation error
\begin{equation}\label{eq:1116-0108}
  \Bigl\|u \cdot \mathbb{I}_{\bigl\{|x| \geqslant \sqrt{2N}/\beta\bigr\}} \Bigr\|
\end{equation}
to be an indicator of the scaled Hermite approximation (projection or interpolation) error. Here, $\|\cdot\|$ denotes the $L^2$ norm, and the indicator function $\mathbb{I}_A$ is defined by
  \begin{equation}\label{eq:indicator}
    \mathbb{I}_A(x)= \begin{cases}1, & x \in A, \\ 0, & \text { otherwise. }\end{cases}
  \end{equation}

So far, our discussion has not gone beyond the scope of previous research.
Based on a similar observation mentioned above,
a scaling factor $\beta = \sqrt{2N}/M$ is suggested by Tang~\cite{tang_hermite_1993} to
scale the collocation interval to $\left[-M, M\right]$ where the exterior error is negligible.

The key point of our idea is that we notice there is a duality between the Hermite approximation of a function $u$ and its Fourier transform
$\mathcal{F}[u]$. Define the Fourier transform as
\begin{equation}\label{eq:1122-0109}
  \mathcal{F}[u](k) = \frac{1}{\sqrt{2\pi}} \int_{-\infty}^{\infty} u(x) e^{-ikx} dx, 
\end{equation}
then we have (see \cref{rmk:1117-0205} and the comment after that)
\begin{equation}\label{eq:1117-0110}
  \mathcal{F}[\sqrt{\beta}\widehat{H}_n(\beta x)](k) = \frac{\left(-i\right)^n}{\sqrt{\beta}} \widehat{H}_n\left(\frac{k}{\beta}\right).
\end{equation}
Hence, approximating $u(x)$ by $\widehat{H}_n(\beta x)$ is the same
as approximating $\mathcal{F}[u](k)$ by $\widehat{H}_n\left({k}/{\beta}\right)$, at least for projection in the $L^2$ norm. Hence if $u$ outside
$\bigl[-\sqrt{2N}/\beta, \sqrt{2N}/\beta\bigr]$ (spatial truncation error), or $\mathcal{F}[u]$ outside $\bigl[-\beta \sqrt{2N}, \beta \sqrt{2N}\bigr]$ 
(frequency truncation error) is not negligible, we cannot expect a good approximation.
In other words, we expect
\begin{equation}\label{eq:1116-0111}
  \bigl\|u \cdot \mathbb{I}_{\{|x| \geqslant \sqrt{2N}/\beta\}}\bigr\| 
  + \bigl\|\mathcal{F}[u](k) \cdot \mathbb{I}_{\{|k| \geqslant \sqrt{2N}\beta\}}\bigr\|
\end{equation}
as an indicator of scaled Hermite approximation. This indicator
precisely characterizes the behavior of the Hermite approximation error, 
which we show in later sections through both theoretical analysis
and numerical experiments.

\subsection{How our results help better understand the Hermite approximation}
According to the previous analysis, if the spatial truncation error or the
frequency truncation error is not negligible, we cannot expect a good
approximation. Hence, finding the optimal scaling factor $\beta$ is equivalent
to balancing the truncation error outside $\bigl[-\sqrt{2N}/\beta, \sqrt{2N}/\beta\bigr]$
in the spatial domain and the truncation error outside $\bigl[-\beta \sqrt{2N}, \beta \sqrt{2N}\bigr]$
in the frequency domain, which guarantees that neither the spatial truncation error nor the frequency truncation error are too large. This simple observation leads
to many amazing results:
\begin{enumerate}
  \item For functions have exponential decay $\exp(-c_1|x|^a)$ in spatial
  domain and exponential decay $\exp(-c_2|k|^b)$ in frequency
  domain, by balancing we find the indicator \cref{eq:1116-0111} can achieve an order of
  $\exp(-c|N|^\gamma)$ using $N+1$ truncated terms, where $\gamma$ satisfy
  \begin{equation*}
    \gamma = \frac{ab}{a + b}.
  \end{equation*}
  From~\cite{boyd_fourier_2014}, we know the Fourier transform of $u = e^{-x^{2n}}, n \in \mathbb{N}^+$
  decay as $\exp(-c|k|^{\frac{2n}{2n-1}})$. Hence \cref{eq:1116-0111} can be improved from $\exp(-c_1|N|^{\frac{n}{2n-1}})$ to
 $\exp(-c_2 N)$ by proper scaling. As we show later, the Hermite approximation error
 (projection error or interpolation error) has a behavior similar to the indicator \cref{eq:1116-0111}.
  Besides, Hardy's uncertainty principle says that a nonzero function and its Fourier transform cannot both decay faster than $\exp(-cx^2)$ \cite{hardy_theorem_1933}, hence the spatial and frequency truncation error cannot both be smaller than $\exp(-cN)$. The above
  analysis implies that the indicator \cref{eq:1116-0111}  seems impossible to achieve
  a super geometric convergence.\\

  \item For functions have algebraic decay $1/(1+x^2)^h$ in the spatial
  domain and exponential decay $\exp(-c|x|^a)$ in the frequency domain
  or vice versa, the classical result shows that without scaling
  the Hermite projection error (we define it later in \cref{eq:def-projHat}) is $N^{-(h-1/4)}$. Our results show that by a proper scaling, the error can be improved to about $N^{-(2h-1/2)}$. That is to say,
  a doubled convergence order can be achieved for this kind of function. \\

  \item For functions have algebraic decay $1/(1+x^2)^{h_1}$ in spatial
  domain and algebraic decay $1/(1+k^2)^{h_2}$ in frequency domain,
  let $a = 2h_1 - 1/2, \, b = 2h_2 - 1/2, \, \gamma = ab/(a+b)$, similar
  to the first type (both exponential decay), the error can be improved
  from $N^{-\min \{a,b\} / 2}$ to $N^{-\gamma}$.\\
  
  \item In \cite{shen_stable_2000},
a sub-geometric convergence when approximating
algebraic decay functions is reported, and in \cite{shen_recent_2009},
it is eventually found that 
this convergence order only holds in the pre-asymptotic range.
That is, the error is $\exp(-c\sqrt{N})$
when using $N+1$ truncated terms, where $N$ is not a large number. This is puzzling,
as the classical error estimate only
predicts a rate of about $N^{-h}$. Through our theory, it
is not difficult to understand this phenomenon. For small $N$,
the approximation error is dominated by the frequency truncation error.
Since the Fourier transform has an exponential decay $\exp(-c|k|)$, the truncation
error outside a frequency interval with upper bound $\mathcal{O}(\sqrt{N})$ is, of course,
$\exp(-c\sqrt{N})$.\\

\end{enumerate}

The remainder of this paper is organized as follows. Our key results are presented in
\cref{sec:main}, more general results in a systematic presentation
are in \cref{sec:general},
experimental results and discussions are in \cref{sec:experiments}, and the conclusions follow in
\cref{sec:conclusions}.

\section{Projection error in $L^2$ norm}
\label{sec:main}
We first give the definition of Hermite projection, then discuss
the duality between a function $u$ and its Fourier transform $\mathcal{F}[u]$
when approximating them by Hermite functions. Afterwards, we establish the
error estimation theorem of the scaled Hermite approximation.

\subsection{Definition of projection}
Let $P_N$ denote the collection of all polynomials of degree no more
than $N$, define $\widehat{P}_N^\beta$ as
\begin{equation}\label{eq:def-PNHat}
  \widehat{P}_N^\beta:=\Bigl\{\phi: \phi=e^{- \beta^2 x^2 / 2} \psi, \quad \forall\, \psi \in P_N\Bigr\} .
\end{equation}
In particular, $\widehat{P}_N$ denotes $\widehat{P}_N^1$. 

We define 
projection $\widehat{\Pi}_N^\beta: L^2(\mathbb{R}) \rightarrow \widehat{P}_N^\beta$ by
\begin{equation}\label{eq:def-projHat}
  \bigl(u-\widehat{\Pi}_N^\beta u, v_N\bigr)=0, \quad \forall\, v_N \in \widehat{P}_N^\beta.
\end{equation}
Similarly, $\widehat{\Pi}_N$ denotes $\widehat{\Pi}_N^1$.


\subsection{The duality of $u$ and $\mathcal{F}[u]$ when approximating by Hermite functions}
We now explain \cref{eq:1117-0110} in detail. To this end, 
we require the following lemmas.

\begin{lemma}\label{lem:cn recurr}
  If $e^{-m x^2} e^{i k x}=\sum\limits_{n=0}^{\infty} c_n \widehat{H}_n$, then $\{c_n\}$ satisfy
  \begin{equation}\label{eq:1117-1}
    c_{n+1}=\frac{1}{2 m + 1} i k \sqrt{\frac{2}{n+1}} c_n-\frac{2 m - 1}{2 m + 1} \sqrt{\frac{n}{n+1}} c_{n-1}.
  \end{equation}
\end{lemma}
  {This lemma can be proved straightforwardly using the three-term recurrence relation~\cref{eq:three-recur}, the derivative relation \cref{eq:Herm-derivative}, and the integration by parts formula.}

\begin{lemma}\label{rmk:1117-0205}
  For Fourier transform $\mathcal{F}[u]$ of $u$ defined by \cref{eq:1122-0109}, we have
  \begin{equation}\label{eq:1117-0110alt}
    \mathcal{F}[\sqrt{\beta}\widehat{H}_n(\beta x)](k) = \frac{\left(-i\right)^n}{\sqrt{\beta}} \widehat{H}_n\left(\frac{k}{\beta}\right).
  \end{equation}
\end{lemma}
This is a known fact (see, for example,
\cite{duoandikoetxea_fourier_2000} and \cite{olver_nist_2010}). It has also been used
to build efficient numerical methods for solving partial differential equations~\cite{he_new_2014}.
This lemma for $\beta=1$ can be proved using \cref{lem:cn recurr} (with $m=0$) and some basics properties of Fourier transforms. A simple rescaling leads to the general case. We omit the details to save space. 

By \cref{rmk:1117-0205}, we immediately have the following result.

\begin{corollary}\label{cor:1122-0203}
For any $L^2$ function $u$, we have
  \begin{equation}
    \bigl\|u - \widehat{\Pi}_N^\beta u \bigr\| = \Bigl\|\mathcal{F}[u] - \widehat{\Pi}_N^{1/\beta} \mathcal{F}[u]\Bigr\|.
  \end{equation}
\end{corollary}
\cref{cor:1122-0203} tells us that approximating $u(x)$ by $\widehat{H}_n(\beta x)$ is the same
as approximating $\mathcal{F}[u](k)$ by $\widehat{H}_n\left(\frac{k}{\beta}\right)$, at least for the projection in the $L^2$ norm.
This inspires us to develop an estimate that includes truncation errors in
both spatial and frequency domains.

\subsection{Projection error in $L^2$ norm: a description without proof}
In the next theorem, we will give an error estimate when using $N+1$
truncated terms of scaled Hermite functions $\widehat{H}_n(\beta x)$
to approximate a function $u$.
This error estimate has three distinct components:
\begin{enumerate}
  \item Spatial truncation error $\bigl\|u \cdot \mathbb{I}_{\{|x| \geqslant a\sqrt{N}/\beta\}}\bigr\|$.
  Here, $\|\cdot\|$ denotes the $L^2$ norm.
  \item Frequency truncation error $\bigl\|\mathcal{F}[u](k) \cdot \mathbb{I}_{\{|k| \geqslant b \sqrt{N} \beta\}}\bigr\|$.
  The Fourier transform $\mathcal{F}[u]$ is defined by \eqref{eq:1122-0109}
  \item Hermite spectral error $\|u\|e^{-cN}$.
\end{enumerate}

\begin{theorem}\label{thm:L2-proj}
  Let $a=b=\frac{1}{2\sqrt{2}},\,c=\frac{1}{16}$, $\|\cdot\|$
  denote the $L^2$ norm, $f \lesssim g$ means $f \leqslant Cg$, where $C$ represents a fixed constant.
  We have
  \begin{equation}
    \begin{aligned}
    \left\|u-\widehat{\Pi}_N^\beta u\right\| 
    \lesssim{} & \bigl\|u \cdot \mathbb{I}_{\{|x|>a \sqrt{N} / \beta\}}\bigr\| 
    + \bigl\|\mathcal{F}[u](k) \cdot \mathbb{I}_{\{|k|>b \sqrt{N} \beta\}}\bigr\| \\
    \quad & +\|u\| e^{-c N}.
    \end{aligned}
  \end{equation}
\end{theorem}

\begin{remark}
  Compared with the classical results, our analysis is more explanatory:
  \begin{enumerate}
    \item Let $\widehat{\partial}_x = \partial_x + x$.
      The classical result says that for fixed $m$, if 
      $\|\widehat{\partial}_x^m u\| < \infty, \, 0 \leqslant m \leqslant N+1$, then
      $\|u - \widehat{\Pi}_N u\|$ is $\mathcal{O}\left(N^{-m/2}\right)$ (see section 7.3.2 of \cite{shen_spectral_2011}).
      This can be understood by our theorem, let $e = u - \widehat{\Pi}_N u$, 
      by repeated use of \cref{lem:1102-lem-inverse} we recover
      \begin{equation}\label{eq:1118-0208}
        \left\|x^m e\right\|, \, \left\|k^m \mathcal{F}[e](k)\right\|=\left\|\partial_x^m e\right\| \lesssim \sum_{k \leqslant m} \left\|\widehat{\partial}_x^k e\right\|.
      \end{equation}
      Noticing that $\widehat{\partial}_x \widehat{H}_n(x) = \sqrt{2n} \widehat{H}_{n-1},\,n\geqslant 1$ and $\widehat{\partial}_x \widehat{H}_0(x) = 0$, by directly calculating the Hermite expansion coefficients, we have
      \begin{equation}
          \sum_{k \leqslant m} \left\|\widehat{\partial}_x^k e\right\| \lesssim \left\|\widehat{\partial}_x^m e \right\| <\infty.
      \end{equation}
      Hence $\left\|x^m e\right\|, \, \left\|k^m \mathcal{F}[e](k)\right\|=\left\|\partial_x^m e\right\| \lesssim \|\widehat{\partial}_x^m u\| < \infty $, which means $e(x)$ and $\mathcal{F}[e](x)$
      both have at least an algebraic decay $x^{-m}$, the spatial and frequency truncation error then satisfy
      \begin{displaymath}
        \begin{aligned}
        \|e \cdot \mathbb{I}_{\{|x|>a \sqrt{N}\}}\| &\lesssim N^{-\frac{m}{2}}\left\|\left(x^m e\right) \cdot \mathbb{I}_{\{|x|>a \sqrt{N}\}}\right\| \\
        &\lesssim N^{-\frac{m}{2}}\left\|x^m e\right\| , \\
        \left\|\mathcal{F}[e](k) \cdot \mathbb{I}_{\{|k|>b \sqrt{N}\}}\right\| &\lesssim N^{-\frac{m}{2}}\left\|\left(k^m \mathcal{F}[e](k)\right) \cdot \mathbb{I}_{\{|k|>b \sqrt{N}\}}\right\| \\
        &\lesssim N^{-\frac{m}{2}}\left\|\partial_x^m e\right\|.
        \end{aligned}
      \end{displaymath}
      Notice that $\bigl\|u - \widehat{\Pi}_N u\bigr\| = \bigl\|e - \widehat{\Pi}_N e\bigr\|$, then by \cref{thm:L2-proj},
      the projection error is, of course, $O\left(N^{-m/2}\right)$.
    \item Since for fixed $m$, the classical result only assumes algebraic decay in the spatial and frequency domains, it cannot indicate an exponential convergence order. Although choosing an $m$ that depends on $N$ can recover an exponential convergence order, $\|\widehat{\partial}_x^m u\|$, in general,
      is difficult to compute. Hence, it is still difficult to predict an exponential convergence order by the classical estimate. By our estimate, this is much easier.
    \item The classical result assumes $\|\widehat{\partial}_x^m u\| < \infty$.
      This condition mixes the information in the spatial and frequency domains. In contrast, our theorem separates spatial and frequency information, making it possible to analyze the impact of the scaling factor.
  \end{enumerate}
\end{remark}

To prove \cref{thm:L2-proj}, we first need to deal with a specific class
of Gaussian-type functions $e^{-(x-s)^2/2 + ikx}$.

\subsection{Projection error for a specific class of Gaussian-type functions}
We establish projection error control for a
specific class of Gaussian-type functions $e^{-(x-s)^2/2 + ikx}$ which will be used
to prove \cref{thm:L2-proj}. 
\begin{lemma}\label{cor:proj-error-Gaussian}
  Let $g_{k, s} (x)=e^{-\frac{1}{2}(x-s)^2+i k x}$, then
 \begin{equation}
    \big\|g_{k, s}-\widehat{\Pi}_N g_{k, s}\big\| \lesssim \frac{1}{\sqrt{(N+1)!}} \cdot\left(\frac{k^2+s^2}{2}\right)^{\frac{N+1}{2}},
  \end{equation}
  with the projection operator
 $\widehat{\Pi}_N$ defined in \cref{eq:def-projHat}.
\end{lemma}
\begin{proof}
  Let $g_{k,s}(x) = \sum\limits_{n=0}^\infty c_n(k,s) \widehat{H}_n(x)$,
  noticing that
  \begin{displaymath}
    g_{k,s}(x) = e^{-s^2/2} \cdot e^{-x^2/2 + i(k-is)x},
  \end{displaymath}
  combining with \cref{rmk:1117-0205} yields (also a direct consequence of (14.9) and (14.10) of \cite{chihara1955generalized})
  \begin{equation}
    c_n(k,s)=\pi^{1 / 4} e^{-\frac{z^2}{4}-\frac{s^2}{2}} \frac{(i z)^n}{\sqrt{2^n n!}},
  \end{equation}
  where $z = k - is$. Hence, by using the Lagrange remainder of Taylor expansion
  \begin{equation}
    \begin{aligned}
    \big\|g_{k, s}-\widehat{\Pi}_N g_{k, s}\big\|^2 & =\sum_{n > N}\left|c_n(k, s)\right|^2 \\
    & =\left|c_0(k,s)\right|^2 \sum_{n > N} \frac{1}{n !}\left(\frac{|z|^2}{2}\right)^n \\
    & =\left|c_0(k,s)\right|^2 \frac{\left(\frac{|z|^2}{2}\right)^{N+1}}{(N+1) !} e^{\theta \frac{|z|^2}{2}} \quad(0 \leqslant \theta \leqslant 1).
    \end{aligned} 
  \end{equation}
  Notice that $c_0(k,s)=\pi^{1 / 4} e^{-\frac{z^2}{4}-\frac{s^2}{2}}$, we have
  \begin{equation}
    \begin{aligned}
    \left\|g_{k, s}-\widehat{\Pi}_N g_{k, s}\right\|^2 & \leqslant \sqrt{\pi} \cdot e^{-\frac{|z|^2}{2}} \cdot \frac{\left(\frac{|z|^2}{2}\right)^{N+1}}{(N+1) !} \cdot e^{\frac{|z|^2}{2}} \\
    & \lesssim \frac{1}{(N+1)!} \left(\frac{k^2+s^2}{2}\right)^{N+1}.
    \end{aligned} 
  \end{equation}
  This ends our proof.
\end{proof}
Now we return to the proof of \cref{thm:L2-proj}.

\subsection{Proof of \cref{thm:L2-proj}}
The idea of the proof is as follows. By properly truncating
the objective function in the spatial and frequency domains,
the remaining part will be a combination of a specific
class of Gaussian type functions, which look like $e^{-(x-s)^2/2+i \xi x}$.
Recall that we deal with these functions in \cref{cor:proj-error-Gaussian}.

Now we prove \cref{thm:L2-proj}.
\begin{proof}
  Let $v(x) = 1/\sqrt{\beta} \cdot u(x/\beta)$. Since 
  \begin{equation}
    \big\|u - \widehat{\Pi}_N^\beta u\big\| = \big\|v - \widehat{\Pi}_N v\big\|,
  \end{equation}
  we only need to consider $\beta = 1$.
  Let $M = \frac{1}{2\sqrt{2}} \sqrt{N}$, consider
  \begin{equation}\label{eq:1030-7}
    h_M(x)=\mathbb{I}_{[-2 M, 2 M]}, \quad G(x)=\frac{1}{\sqrt{2\pi}} e^{-x^2/2}, \quad T_M=h_M * G.
  \end{equation}
  Let $u_M = u \cdot T_M$, then
  \begin{equation}\label{eq:1030-8}
    \begin{aligned}
    \bigl\|u-\widehat{\Pi}_N u\bigr\| 
    & \leqslant\left\|u-u_M\right\| +\bigl\|u_M-\widehat{\Pi}_N u_M\bigr\|+\bigl\|\widehat{\Pi}_N u_M-\widehat{\Pi}_N u\bigr\| \\
    & \lesssim\left\|u-u_M\right\|+\bigl\|u_M-\widehat{\Pi}_N u_M\bigr\| \\
    & \triangleq E_1+E_2.
    \end{aligned}
  \end{equation}
  For $E_1$ we have
  \begin{equation}\label{eq:1030-9}
    \begin{aligned}
    \left\|u-u_M\right\| & \leqslant\left\|\left(u-u_M\right) \cdot \mathbb{I}_{\{|x| \leqslant M\}}\right\|+\left\|\left(u-u_M\right) \cdot \mathbb{I}_{\{|x|>M\}}\right\| \\
    & \triangleq E_{11}+E_{12}.
    \end{aligned}
  \end{equation}
  For $E_{11}$ we have
  \begin{equation}\label{eq:1030-10}
    \begin{aligned}
    E_{11} & = \left\|\left(u-u_M\right) \cdot \mathbb{I}_{\{|x| \leqslant M\}}\right\| \\
    & =\left(\int_{|x| \leqslant M} u^2\left(T_M-1\right)^2 d x\right)^{1 / 2} \\
    & \leqslant \| \left(T_M-1\right) \cdot \mathbb{I}_{\{|x| \leqslant M\}}\left\|_{\infty}\right\| u \|.
    \end{aligned}
  \end{equation}
  Notice that when $|x| \leqslant M$, we have
  \begin{equation}\label{eq:1030-11}
    \begin{aligned}
    \left|T_M-1\right| & =\left|\int_{-2 M}^{2 M} \frac{1}{\sqrt{2\pi}} e^{-\cdot (s-x)^2/2} d s-1\right| \\
    & \leqslant \int_{|s|>M} \frac{1}{\sqrt{2\pi}} e^{-s^2/2} d s \\
    & \lesssim e^{- M^2/2}=e^{-\frac{1}{16} N}.
    \end{aligned}
  \end{equation}
  Putting \cref{eq:1030-11} back into \cref{eq:1030-10} yields
  \begin{equation}\label{eq:1030-12}
    \left\|\left(u-u_M\right) \cdot \mathbb{I}_{\{|x| \leqslant M\}}\right\| \lesssim e^{-\frac{1}{16} N}\|u\|.
  \end{equation}
  As for $E_{12}$ we have
  \begin{equation}\label{eq:1030-13}
    \begin{aligned}
    E_{12} & = \left\|(u - u_M) \cdot \mathbb{I}_{\{|x|>M\}}\right\| \\
    & \leqslant\|u \cdot \mathbb{I}_{\{|x|>M\}}\|+\left\|u_M \cdot \mathbb{I}_{\{|x|>M\}}\right\| \\
    & \lesssim\|u \cdot \mathbb{I}_{\{|x|>M\}}\|.
    \end{aligned}
  \end{equation}
  Combining \cref{eq:1030-12}, \cref{eq:1030-13} with \cref{eq:1030-9} yields
  \begin{equation}\label{eq:1030-14}
    \left\|u-u_M\right\| \lesssim e^{-\frac{1}{16} N}\|u\|+\|u \cdot \mathbb{I}_{\{|x|>M\}}\|.
  \end{equation}
  Let $B = M = \frac{1}{2\sqrt{2}} \sqrt{N}$, consider
  \begin{equation}\label{eq:1030-15}
    u^B=\frac{1}{\sqrt{2 \pi}} \int_{|k| \leqslant B} \mathcal{F}[u](k) e^{i k x} d k, \quad u_M^B=u^B T_M,
  \end{equation}
  we have
  \begin{equation}\label{eq:1030-16}
    \begin{aligned}
    \left\|u_M-\widehat{\Pi}_N u_M\right\| & \leqslant\left\|u_M-u_M^B\right\|+\left\|u_M^B-\widehat{\Pi}_N u_M^B\right\|+\left\|\widehat{\Pi}_N u_M^B-\widehat{\Pi}_N u_M\right\| \\
    & \lesssim\left\|u_M-u_M^B\right\|+\left\|u_M^B-\widehat{\Pi}_N u_M^B\right\| \\
    & \triangleq E_{21} + E_{22}.
    \end{aligned}
  \end{equation}
  For $E_{21}$ we have
  \begin{equation}\label{eq:1030-17}
    \begin{aligned}
    E_{21} & =\left\|\left(u-u^B\right) \cdot T_M\right\| \\
    & \leqslant\left\|u-u^B\right\| \\
    & =\|\mathcal{F}[u](k) \cdot \mathbb{I}_{\{|k|>B\}}\|.
    \end{aligned}
  \end{equation}
  Consider
  \begin{equation}\label{eq:1030-20}
    g_{k, s}(x)=e^{-(x-s)^2/2+i k x}=\sum_n c_n(k, s) \widehat{H}_n(x).
  \end{equation}
  Let $\Omega=\{(k, s):|k| \leqslant B, |s| \leqslant 2 M\}$,
  we have
  \begin{equation}
    \begin{aligned}
    u_M^B & =u^B \cdot T_M \\
    & =\frac{1}{2 \pi} \int_{\Omega} \mathcal{F}[u](k) e^{- (x-s)^2/2 +i k x} d k d s.
    \end{aligned}
  \end{equation}
  Then by \cref{eq:1030-20}, we recover
  \begin{equation}\label{eq:1030-18}
    \begin{aligned}
    u_M^B & =\frac{1}{2 \pi} \int_{\Omega} \mathcal{F}[u](k) \left(\sum_n c_n(k, s) \widehat{H}_n(x)\right) d k d s \\
    & =\sum_n \left(\frac{1}{2 \pi} \int_{\Omega} \mathcal{F}[u](k) c_n(k, s) d k d s\right) \widehat{H}_n(x) \\
    & \triangleq \sum_n c_n \widehat{H}_n(x).
    \end{aligned}
  \end{equation}
  We have
  \begin{equation}
    \begin{aligned}
    \left\|u_M^B-\widehat{\Pi}_N u_M^B\right\|^2 & =\sum_{n>N}\left|c_n\right|^2 \\
    & = \sum_{n>N} \frac{1}{4 \pi^2} \left(\int_{\Omega} \mathcal{F}[u](k) c_n(k, s) d k d s\right)^2.
    \end{aligned}
  \end{equation}
  Using Cauchy-Schwarz inequality yields
  \begin{equation}\label{eq:1030-22}
    \begin{aligned}
    \left\|u_M^B-\widehat{\Pi}_N u_M^B\right\|^2 & \lesssim  \int_{\Omega}|\mathcal{F}[u](k)|^2 d k d s \cdot \int_{\Omega} \sum_{n>N} \left|c_n(k, s)\right|^2 d k d s \\
    & \lesssim 4 M\|u\|^2 \int_{\Omega} \left\|g_{k, s}-\widehat{\Pi}_N g_{k, s}\right\|^2 d k d s .
    \end{aligned}
  \end{equation}
  Let $D=\left\{(k, s): k^2+s^2 \leqslant B^2+(2 M)^2\right\}$, by \cref{cor:proj-error-Gaussian}, 
  \begin{equation}\label{eq:1207-2}
    \begin{aligned}
    \int_{\Omega}\left\|g_{k, s}-\widehat{\Pi}_N g_{k, s}\right\|^2 d k d s 
    & \lesssim \int_{\Omega} \frac{1}{(N+1)!}\left(\frac{k^2+s^2}{2}\right)^{N+1} d k d s \\
    & \lesssim \frac{1}{(N+1)!} \int_D\left(\frac{k^2+s^2}{2}\right)^{N+1} d k d s.
    \end{aligned}
  \end{equation}
  Let $R=\sqrt{B^2+(2 M)^2}=\frac{\sqrt{5}}{2 \sqrt{2}} \sqrt{N}$,
  \begin{equation}\label{eq:1207-3}
    \begin{aligned}
    \int_D\left(\frac{k^2+s^2}{2}\right)^{N+1} d k d s & =\int_0^{2 \pi} d \theta \int_0^R\left(\frac{r^2}{2}\right)^{N+1} r d r \\
    & =2 \pi \cdot 2^{-(N+1)} \cdot \frac{R^{2 N+4}}{2 N+4}.
    \end{aligned}
  \end{equation}
  Combining $n!>\sqrt{2 \pi} n^{n+1 / 2} e^{-n}$ with \cref{eq:1207-2}, \cref{eq:1207-3} yields
  \begin{equation}\label{eq:1207-4}
    \int_{\Omega}\bigl\|g_{k, s}-\widehat{\Pi}_N g_{k, s}\bigr\|^2 d k d s \lesssim N^{-1/2} \cdot e^{-\frac{N}{8}}.
  \end{equation}
  Putting \cref{eq:1207-4} into \cref{eq:1030-22} yields
  \begin{equation}\label{eq:1030-23}
    \left\|u_M^B-\widehat{\Pi}_N u_M^B\right\| \lesssim\|u\| e^{-\frac{1}{16} N}.
  \end{equation}
  Combining \cref{eq:1030-17}, \cref{eq:1030-23} with \cref{eq:1030-16} yields
  \begin{equation}\label{eq:1030-24}
    \bigl\|u_M-\widehat{\Pi}_N u_M\bigr\| \lesssim\|\mathcal{F}[u](k) \cdot \mathbb{I}_{\{|k|>B\}}\|+\|u\| e^{-\frac{1}{16} N}.
  \end{equation}
  Putting \cref{eq:1030-14}, \cref{eq:1030-24} back into \cref{eq:1030-8} yields
  \begin{equation}\label{eq:1030-25}
    \begin{aligned}
    \bigl\|u-\widehat{\Pi}_N u\bigr\| & \lesssim\Bigl\|u \cdot \mathbb{I}_{\left\{|x|>\frac{1}{2\sqrt{2}} \sqrt{N}\right\}}\Bigr\| \\
    & +\Bigl\|\mathcal{F}[u](k) \cdot \mathbb{I}_{\left\{|k|>\frac{1}{2\sqrt{2}} \sqrt{N}\right\}}\Bigr\| \\
    & +\|u\| e^{-\frac{1}{16} N},
    \end{aligned}
  \end{equation}
  this ends the proof of \cref{thm:L2-proj}.
\end{proof}

\section{Projection error with higher-order derivatives and interpolation error}
\label{sec:general}
We now consider the error estimate for $\bigl\|\partial_x^l \bigl(
  u - \widehat{\Pi}_N^{\beta} u \bigr)\bigr\|$. It is not difficult
to imagine that $\bigl\|\partial_x^l \bigl(u - \widehat{\Pi}_N^{\beta} u \bigr)\bigr\|$
can be controlled by the spatial truncation error, the frequency truncation
error, and the Hermite spectral error of $\partial_x^j u$, where $0 \leqslant j \leqslant l$.
The only difference is that an amplification factor from the inverse inequality will be introduced.

We first introduce inverse inequalities that will be used later and
then give estimates for projection error with high-order derivatives.
Interpolation errors are also considered. Since there is no
essential difficulty here, we will omit the detailed proof and
give only the conclusions.
\subsection{Inverse inequalities}
By Lemma 2.2 of \cite{guo_error_1999} we have the following lemma:
\begin{lemma}\label{thm:1102-inverse}
  Let $\widehat{\partial}_x = \partial_x + x$. For any
  $\Psi \in \widehat{P}_N$ (defined in \cref{eq:def-PNHat}), we have
  \begin{equation}
    \bigl\|\widehat{\partial}_x \Psi\bigr\| \lesssim \sqrt{N}\|\Psi\| .
  \end{equation}
\end{lemma}

By (B.36b) of \cite{shen_spectral_2011} we have the following lemma:
\begin{lemma}\label{lem:1102-lem-inverse}
  Let $\widehat{\partial}_x = \partial_x + x$, then
  \begin{equation}
    \|xu\| \leqslant \|u\| + \|\widehat{\partial}_x u\|. 
  \end{equation}
\end{lemma}

Combining \cref{thm:1102-inverse} with \cref{lem:1102-lem-inverse}
we have the following result.
\begin{corollary}
  For any $\Psi \in \widehat{P}_N$ (defined in \cref{eq:def-PNHat})
  and $l = 0,1,2$
  \begin{equation}
    \left\|\partial_x^l \Psi\right\| \lesssim N^{\frac{l}{2}}\|\Psi\|.
  \end{equation}
\end{corollary}
\begin{proof}
  The case of $l = 0$ is obvious.

  For $l = 1$, we have
  \begin{equation}\label{eq:1118-0304}
    \begin{aligned}
    \left\|\partial_x \Psi\right\| & \leqslant\|x \Psi\|+\bigl\|\widehat{\partial}_x \Psi\bigr\| \\
    & \lesssim\|\Psi\|+\bigl\|\widehat{\partial}_x \Psi\bigr\| \quad \left(\cref{lem:1102-lem-inverse}\right)\\
    & \lesssim \sqrt{N}\|\Psi\| \quad \left(\cref{thm:1102-inverse}\right).
    \end{aligned}
  \end{equation}
  The case of $l=2$ can be proved recursively.
\end{proof}

\subsection{Projection error with higher-order derivatives}
The following theorem can be proved following the same
procedure in \cref{sec:main} and using the inverse inequality.

Let $E_{s,N}^\beta(u), E_{f,N}^\beta(u), E_{h,N}(u)$ denote the 
spatial truncation error, frequency truncation error, and
Hermite spectral error of $u$, i.e.,
\begin{equation}
  \begin{aligned}
  & E_{s,N}^\beta(u)=\Bigl\|u \cdot \mathbb{I}_{\left\{|x|>\frac{1}{2\sqrt{2}} \sqrt{N} / \beta\right\}}\Bigr\|, \\
  & E_{f,N}^\beta(u)=\Bigl\|\mathcal{F} [u] (k) \cdot \mathbb{I}_{\left\{|k|>\frac{1}{2\sqrt{2}} \sqrt{N} \beta\right\}}\Bigr\|, \\
  & E_{h,N}(u)=\|u\| e^{-\frac{1}{16} N} .
  \end{aligned}
\end{equation}
Then define
\begin{equation}
  E_N^\beta(u) = E_{s,N}^\beta(u) + E_{f,N}^\beta(u) + E_{h,N}(u).
\end{equation}
We have
\begin{theorem}\label{thm:1117-0304}
  \begin{equation}
    \begin{aligned}
      \left\|u-\widehat{\Pi}_N^\beta u\right\| &\lesssim E_N^\beta(u), \\
      \left\|\partial_x\bigl(u-\widehat{\Pi}_N^\beta u\bigr)\right\| &\lesssim E_N^\beta\left(\partial_x u\right)+\beta \sqrt{N} E_N^\beta(u), \\
      \left\|\partial_x^2\bigl(u-\widehat{\Pi}_N^\beta u\bigr)\right\| &\lesssim E_N^\beta\left(\partial_x^2 u\right)+\beta\sqrt{N} E_N^\beta(\partial_x u)+\beta^2 N E_N^\beta(u).
    \end{aligned}
  \end{equation}
\end{theorem}

\subsection{Interpolation error}
Let $\left\{x_j\right\}_{j=0}^N$ be the roots of $\widehat{H}_{N+1}(x)$, define the interpolation
operator $\widehat{I}_N^\beta: L^2(\mathbb{R}) \rightarrow \widehat{P}_N^\beta$ by
\begin{equation}\label{eq:def-hatInterpolation}
  u\left(x_j / \beta\right) = \widehat{I}_N^\beta [u] \left(x_j / \beta\right), \quad \forall\, 0 \leqslant j \leqslant N.
\end{equation}
In particular, $\widehat{I}_N = \widehat{I}_N^1$.

The procedure for establishing the interpolation error from the projection
error is the same as in \cite{aguirre_hermite_2005}. We only list a key lemmas and the final
results.

By Theorem 6 of \cite{aguirre_hermite_2005} we obtain the following lemma:
\begin{lemma}\label{lem:hatInterpolation-stability}
  \begin{equation}\label{eq:hat-int-stable}
    \left\|\widehat{I}_N u\right\| \lesssim N^{\frac{1}{6}}\|u\|+N^{-\frac{1}{3}}\|\partial_x u\|.
  \end{equation}
\end{lemma}

\begin{theorem}\label{thm:1117-0306}
  \begin{equation}
    \begin{aligned}
    \left\|u-\widehat{I}_N^\beta u\right\| &\lesssim \beta^{-1} N^{-\frac{1}{3}} E_N^\beta\left(\partial_x u\right)+N^{\frac{1}{6}} E_N^\beta(u), \\
    \left\|\partial_x\left(u-\widehat{I}_N^\beta u\right)\right\| &\lesssim N^{\frac{1}{6}} E_N^\beta\left(\partial_x u\right)+\beta N^{\frac{2}{3}} E_N^\beta(u), \\
    \left\|\partial_x^2\left(u-\widehat{I}_N^\beta u\right)\right\| &\lesssim E_N^\beta\left(\partial_x^2 u\right)+\beta N^{\frac{2}{3}} E_N^\beta\left(\partial_x u\right)+\beta^2 N^{\frac{7}{6}} E_N^\beta(u).
    \end{aligned}
  \end{equation}
\end{theorem}

\subsection{Application: optimality of scaled Gauss--Hermite quadrature}
\label{sec: Hermite optimality}
Consider the scaled Gauss--Hermite quadrature
(taking $\beta = 1$ in \cref{eq: Hermite quadrature} yields a standard Gauss--Hermite quadrature)
\begin{equation}\label{eq: Hermite quadrature}
  Q_N^\beta (h_1(x) h_2(x)) = \int_{-\infty}^{\infty} \widehat{I}_N^\beta (h_1) \widehat{I}_N^\beta (h_2) dx,
\end{equation}
where the interpolation operator $\widehat{I}_N^\beta$ is defined by \cref{eq:def-hatInterpolation}, then
with proper $\beta$, the Gauss--Hermite quadrature can be comparable to other methods.
Notice that the error between $Q_N^\beta (h_1 h_2)$ and exact value can be controlled by
$\|h_1 - \widehat{I}_N^\beta (h_1)\|$ and $\|h_2 - \widehat{I}_N^\beta (h_2)\|$, then
by \cref{thm:L2-proj} and \cref{lem:hatInterpolation-stability} we can establish
an estimate.

Kazashi et al.~\cite{kazashi_suboptimality_2023} compared the Gauss--Hermite quadrature
with the trapezoidal rule. They argued that Gauss--Hermite quadrature is only ``sub-optimal'',
since for function $f$ with $\alpha$ order smoothness in some sense, only an order about $N^{-\alpha/2}$ with $N$ function evaluations by Gauss--Hermite quadrature
$Q_N^1(f e^{-x^2})$ can be achieved,
while a suitably truncated trapezoidal rule achieves about $N^{-\alpha}$ up to a logarithmic factor.
However, since the trapezoidal rule requires $f e^{-x^2}$ to have exponential decay $e^{-px^2}$ with $p >0$, then under the same
conditions, let us denote:
\begin{displaymath}
  h_1(x) = f e^{-(1-p/2) x^2},\quad h_2(x) = e^{-p/2 x^2}.
\end{displaymath}
Since $h_1(x)$ has exponential decay, by $\alpha$ order smoothness, its Fourier
transform $\mathcal{F}[h_1](\xi)$ at least has algebraic decay similar to $|\xi|^{-\alpha}$. 
Let $\beta = C \sqrt{N} / \sqrt{\ln N}$, 
by \cref{thm:L2-proj} and \cref{lem:hatInterpolation-stability}, 
$\|h_1 - \widehat{I}_N^\beta (h_1)\|$, $\|h_2 - \widehat{I}_N^\beta (h_2)\|$
and the scaled Gauss--Hermite quadrature $Q_N^\beta (f e^{-x^2}) = Q_N^\beta (h_1 h_2)$ defined in \cref{eq: Hermite quadrature} can achieve a convergence rate of about $N^{-\alpha}$ up to a logarithmic factor.

For functions decaying at the rate $\exp(-\tilde{C}|x|^\rho)\,\,(\rho \geq 1)$ on the real axis under other suitable assumptions, 
Sugihara~\cite{sugihara1997optimality} established the rate
$\exp(-C N^{\rho / (\rho+1)})$ by trapezoidal rule, while the existing literature only shows that the Hermite approximation can achieve a rate $\exp(-C\sqrt{N})$. Now, we show that by proper scaling, the convergence rate $\exp(-C N^{\rho / (\rho+1)})$ can be recovered by a scaled Hermite approximation.
Define
\begin{displaymath}
  h_1(x) = f(x) / \sqrt{\omega(x)}, \quad h_2(x) = \sqrt{\omega(x)},
\end{displaymath}
where $\omega(x)$ is defined in Theorem 3.1 of~\cite{sugihara1997optimality}.
Then $h_1(x), h_2(x)$ both have $\exp(-\tilde{C}|x|^\rho)$ decay, and
by Theorem 1.8.4 of Stenger's book~\cite{stenger1993numerical}, Fourier transforms
of $h_1(x), h_2(x)$ decay at least as $\exp(-c|\xi|)$. Using decay information in spatial and frequency domains, by \cref{thm:L2-proj} and \cref{lem:hatInterpolation-stability}, 
one can find without scaling $\|h_1 - \widehat{I}_N^\beta (h_1)\|$, $\|h_2 - \widehat{I}_N^\beta (h_2)\|$ and the convergence rate of the Gauss--Hermite quadrature $Q_N^\beta(f) = Q_N^\beta(h_1 h_2)$
defined in \cref{eq: Hermite quadrature} is at least $\exp(-C\sqrt{N})$, while by a proper
scaling the rate can also achieve $\exp(-CN^{\rho/(\rho+1)})$, which is comparable to the trapezoidal rule. Given that Goda et al. indicated that this convergence order is a sharp upper bound for the worst case~\cite{goda2024sharp}, our argument demonstrates that the Hermite quadrature likewise achieves optimality.

Applying a similar argument to the functions investigated by
Trefethen~\cite{trefethen2022exactness}, which are analytic in
a strip with $\exp(-x^2)$ decay on the real axis,
the scaled Gauss--Hermite quadrature error can be improved to $\exp(-CN^{2/3})$ by proper scaling, comparable to domain truncation methods. 
The above argument holds for
$f(x) = \cos(x^3)$, which is mentioned in a numerical experiment in~\cite{trefethen2022exactness}.


\section{Numerical results and discussions}
\label{sec:experiments}
In this section, we validate our error estimates with numerical examples on
quadrature, interpolation, and solving the model equation in the following.
Some typical examples are selected to show the new insights gained from
our estimates.

\subsection{Model problem}
Consider the model problem:
\begin{equation}\label{eq:model-problem}
  -u_{x x}+\gamma u=f,\quad  x \in \mathbb{R}, \gamma>0 ; \quad \lim _{|x| \rightarrow \infty} u(x)=0 \text {. }
\end{equation}
The scaled Hermite--Galerkin method for \cref{eq:model-problem} is
\begin{equation}\label{eq:Hemrite-Galerkin}
  \left\{\begin{array}{l}
\text { Find } u_N \in \widehat{P}_N^{\beta} \text { such that } \\
\left(\partial_x u_N, \partial_x v_N\right)+\gamma\left(u_N, v_N\right)=\left(\widehat{I}_N^\beta f, v_N\right), \quad \forall v_N \in \widehat{P}_N^{\beta},
\end{array}\right.
\end{equation}
where $\widehat{I}_N^{\beta}$ is the modified Hermite--Gauss interpolation operator.

The error of the numerical solution can be controlled by projection
error and interpolation error as
\begin{equation}\label{eq:cea}
  \left\|u - u_N\right\|_1 \leqslant C_\gamma \left(\bigl\|u - \widehat{\Pi}_N^{\beta}u\bigr\|_1+\bigl\|f-\widehat{I}_N^{\beta} f\bigr\|\right).
\end{equation}
\cref{eq:cea} can be proved by the same argument as Theorem 7.19 of \cite{shen_spectral_2011}.
Since the projection and interpolation error estimates have been established in \cref{thm:1117-0304} and \cref{thm:1117-0306}, 
we get the estimate for $\|u - u_N\|_1$.
\begin{theorem}
  For model problem \cref{eq:model-problem}, the solution of \cref{eq:Hemrite-Galerkin}
  satisfies
  \begin{equation}
    \begin{aligned}
    \left\|u-u_N\right\|_1 & \lesssim E_N^\beta(\partial_x u)+(1+\beta \sqrt{N}) E_N^\beta(u) \\
    & +\beta^{-1} N^{-\frac{1}{3}} E_N^\beta(\partial_x f)+N^{\frac{1}{6}} E_N^\beta(f),
    \end{aligned}
  \end{equation}
  where $\lesssim$ means $\leqslant C_\gamma$ here, $C_\gamma$ is a constant depends on $\gamma$.
\end{theorem}

\subsection{The optimal scaling balances the spatial and frequency truncation error}
From \cref{thm:L2-proj} we know that the projection error can be controlled
by spatial, frequency truncation error, and Hermite spectral error. The
Hermite spectral error, which takes the form $\|u\|e^{-cN}$, is independent
of scaling, therefore, we only need to consider these truncation errors.

If the spatial and frequency truncation error imbalance, i.e., one
is much larger than the other (often by an order of magnitude), 
then by proper scaling we can reduce the error of the dominant side
and thus reduce the total error.

In general, finding the optimal scaling is equivalent to
balancing the spatial and frequency truncation error. In practice, 
we can adopt a strategy similar to the adaptive finite element method:
gradually increasing the number of truncated terms while adjusting
the scaling factor to balance the errors; see \cref{alg:adaptive scale}.

\begin{algorithm}
\caption{Adaptive scaling}
\label{alg:adaptive scale}
\begin{algorithmic}
\STATE{Define $N = N_0,\, \beta = \beta_0$}
\WHILE{$N < N_\text{max}$}
\STATE{Compute spatial and frequency error $E_s(N),\,E_f(N)$}
\IF{$E_s(N) < E_f(N)$}
\STATE{$N = 2N,\, \beta = \sqrt{2}\beta$}
\ELSE
\STATE{$N = 2N,\, \beta = \beta/\sqrt{2}$}
\ENDIF
\ENDWHILE
\RETURN $N,\,\beta$
\end{algorithmic}
\end{algorithm}

Given that this article focuses primarily on theoretical aspects, 
\cref{alg:adaptive scale} is merely a preliminary proposal. 
We will defer the development of a more efficient and
practical adaptive scaling method to a subsequent publication.

Returning to the numerical experiments, in the following examples we will demonstrate how the balancing
principle works.

\subsection{Properly scaled Hermite quadrature achieves optimality}
In this subsection, we verify the validity of the theory presented in
\cref{sec: Hermite optimality}. Take $u = e^{-x^2}/(1+x^8)$ as an example,
let $h_1(x) = e^{-x^2/2}/(1+x^8),\,h_2(x) = e^{-x^2/2}$, then $u=h_1h_2$, we know that
the error between the scaled Hermite quadrature
\begin{displaymath}
  Q_N^\beta (h_1(x) h_2(x)) = \int_{-\infty}^{\infty} \widehat{I}_N^\beta (h_1) \widehat{I}_N^\beta (h_2) dx
\end{displaymath}
and the exact value of the integral $Q(u(x)) = \int_{-\infty}^\infty u(x) dx$
can be controlled by $\|h_1 - \widehat{I}_N^\beta (h_1)\|$ and $\|h_2 - \widehat{I}_N^\beta (h_2)\|$.
Since $h_1(x)$ is an analytic function with certain properties in a strip domain,
by Theorem 1.8.4 of Stenger's book~\cite{stenger1993numerical}, we know that
\begin{equation}\label{eq:0905-1}
  \mathcal{F}[h_1](k) \lesssim e^{-ck}.
\end{equation}
Combining \cref{eq:0905-1} and \cref{thm:L2-proj} yields that for $\beta = 1$
\begin{equation}\label{eq:0905-2}
  \left|Q(u) - Q_N^\beta(u)\right| \lesssim e^{-c\sqrt{N}},
\end{equation}
while for $\beta = N^{\frac{1}{6}}$,
\begin{equation}\label{eq:0905-3}
  \left|Q(u) - Q_N^\beta(u)\right| \lesssim e^{-cN^{2/3}}.
\end{equation}
Since the largest quadrature node of the standard Hermite rule grows on the order of
$\sqrt{N}$, and the optimal scaling factor is $\beta = N^{\frac{1}{6}}$, the nodes
of the scaled Hermtie quadrature are distributed over the interval
$\left[-LN^{\frac{1}{3}},LN^{\frac{1}{3}}\right]$. As we mentioned earlier,
Trefethen~\cite{trefethen2022exactness} also pointed out that for functions analytic in
a strip with $\exp(-x^2)$ decay on the real axis, the Gauss-Legendre, Clenshaw--Curtis and trapezoidal
quadrature can achieve the same convergence rate of $\exp(-CN^{2/3})$ by applying them on
the same truncated interval $\left[-LN^{\frac{1}{3}},LN^{\frac{1}{3}}\right]$.

In numerical experiments,
we choose $\beta$ so that the $N$ nodes
of the scaled Hermite quadrature are distributed over the interval
$\left[-N^{\frac{1}{3}},N^{\frac{1}{3}}\right]$. For the Gauss--Legendre, Clenshaw--Curtis, and trapezoidal
quadrature, we employ them in the same interval with $N$ quadrature points.
\cref{fig:quadrature-cmp} presents the results of the numerical experiments.
\begin{figure}[H]
  \centering
  \includegraphics[width=0.6\linewidth]{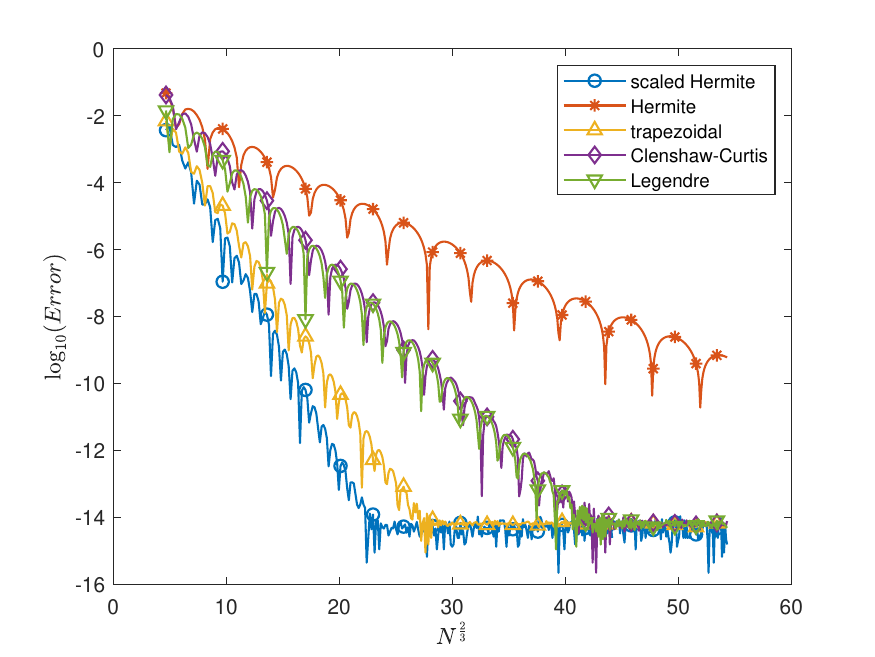}
  \caption{Errors for different quadrature rules to integrate $u=e^{-x^2}\frac{1}{1+x^8}$.  The number of quadrature nodes N varies from 10 to 400.}
  \label{fig:quadrature-cmp}
\end{figure}
The numerical results confirm \cref{eq:0905-3} and demonstrate that 
the properly scaled Hermite quadrature rule achieves efficiency 
comparable to that of a appropriately truncated rule on a bounded domain. 
Moreover, it may even have advantages in certain scenarios.

\subsection{Proper scaling recovers a geometric convergence}
\label{subsec: exp conv}
From (59), (60) of \cite{boyd_fourier_2014}, we know that the Fourier transform of $u = e^{-x^{2n}}$ satisfies
\begin{displaymath}
  \begin{aligned}
  & \Phi(k ; n) \sim\left(\frac{k}{2 n}\right)^{1 /(2 n-1)} \exp \left(z \Psi\left(t_\sigma\right)\right) \frac{\sqrt{\pi}}{\sqrt{z}} \frac{1}{\sqrt{-P_2}} \\
  & =C_1 k^{(1-n) /(2 n-1)} \exp \left(-C_2 k^{2 n /(2 n-1)}\right) \cos \left(C_3 k^{2 n /(2 n-1)}-\xi_n\right),
  \end{aligned}
\end{displaymath}
where $C_1, C_2, C_3, \xi_n$ are constants that depend on $n$.
Since
\begin{displaymath}
  u(x) \lesssim e^{-x^{2n}}, \quad \mathcal{F}[u](k) \lesssim e^{-c k^{2n/(2n-1)}},
\end{displaymath}
by \cref{thm:L2-proj} we know if $\beta = a$ where $a$ is a constant, then
\begin{equation}\label{eq:exp(-x2n)-beta1}
  \|u - \widehat{\Pi}_N^\beta u\| \lesssim e^{-c N^{n/(2n-1)}}.
\end{equation}
Taking $\beta = a \left(N\right)^{(n-1)/(2n)}$ balances the spatial
and frequency truncation error, hence
\begin{equation}\label{eq:exp(-x2n)-beta-opt}
  \|u - \widehat{\Pi}_N^\beta u\| \lesssim e^{-c N}.
\end{equation}

We compute the interpolation error to verify the aforementioned convergence order.
Taking $u = e^{-x^8}$ as an example, with scaling factor
$\beta = 1$, $N$ ranges from $5$ to $500$, the $L^2$ error
$\|u - \widehat{I}_N^\beta u\|$ is presented in \cref{fig:exp(-x8)-beta1}.
Recall that the interpolation operator $\widehat{I}_N^\beta$ is defined in
\cref{eq:def-hatInterpolation}.
Despite an oscillation, we observe the expected convergence order as in 
\cref{eq:exp(-x2n)-beta1}.

\begin{figure}[htbp]
\centering
\subfigure[$\|u - \widehat{I}_N^\beta u\|$, $\beta=1$]{
  \begin{minipage}[t]{0.45\textwidth}
  \centering
  \includegraphics[width=1\textwidth,trim=15 5 20 0, clip]{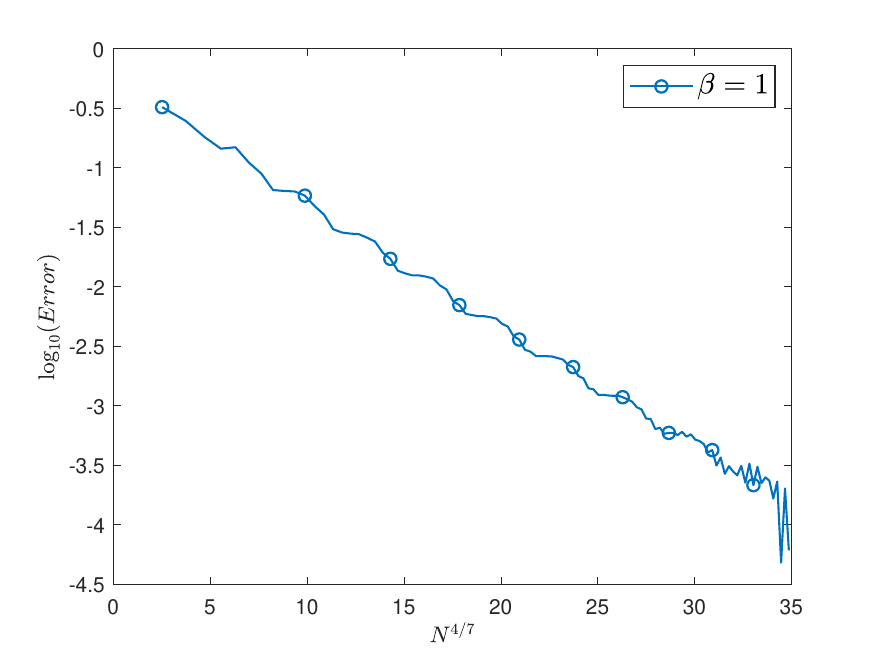}
  \label{fig:exp(-x8)-beta1}
  \end{minipage}
}
\subfigure[$\|u - \widehat{I}_N^\beta u\|$, $\beta=N^{3/8}$]{
  \begin{minipage}[t]{0.45\textwidth}
  \centering
  \includegraphics[width=1\textwidth,trim=15 5 20 0, clip]{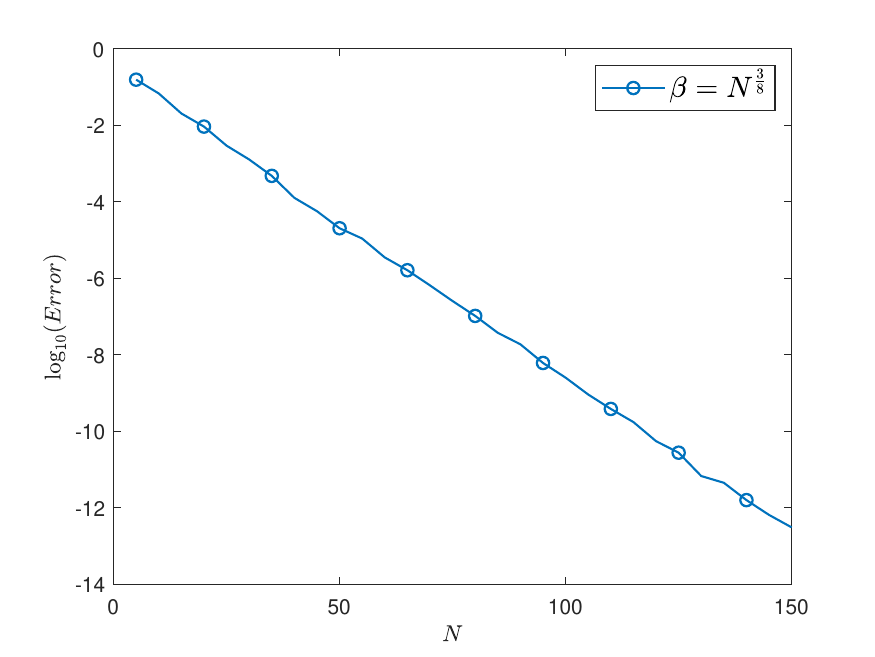}
  \label{fig:exp(-x8)-beta-opt}
  \end{minipage}
}
\caption{Interpolation error without scaling and with optimal scaling for $u=e^{-x^8}$.}
\end{figure}

The $L^2$ error $\|u - \widehat{I}_N^\beta u\|$ corresponding to scaling $\beta = N^{3/8}$
is given in \cref{fig:exp(-x8)-beta-opt} with $N$ ranging from $5$ to $150$, which verifies \cref{eq:exp(-x2n)-beta-opt}.

We believe that the geometric convergence in \cref{eq:exp(-x2n)-beta-opt}
is not a coincidence. It seems that the spatial and frequency truncation errors cannot both achieve
a super-geometric convergence. This can be explained by Hardy's
uncertainty principle \cite{hardy_theorem_1933}:
\begin{theorem}[Hardy's uncertainty principle]\label{thm:uncertainty}
  Let $a,b>0$. Assume that:
  \begin{equation}
    \begin{aligned}
      \left|u(x)\right| &\lesssim e^{-ax^2},\\
      \left|\mathcal{F}[k](u)\right| &\lesssim e^{-bk^2}.
    \end{aligned}
  \end{equation}
  If $ab>\frac{1}{4}$, then $u \equiv 0$, if $ab=\frac{1}{4}$, 
  then $u = C e^{-a x^2}$.
\end{theorem}
From \cref{thm:uncertainty} we know that the spatial and frequency
truncation errors introduced in \cref{thm:L2-proj}
cannot both have a super-geometric convergence. 
Therefore, even if the Hermite spectral error is ignored, the estimation in \cref{thm:L2-proj} can at most guarantee a
geometric convergence. Despite this, the possibility of recovering a geometric convergence
is still good news for us.

\subsection{Proper scaling doubles the convergence order}
\label{subsec: double conv}
In this subsection, we consider the approximation of
$u = \left(1+x^2\right)^{-h}$, which stands for a class
of functions having algebraic decay in the spatial domain
and exponential decay in the frequency domain.

Since 
\begin{equation}\label{eq:1104-0407}
  \mathcal{F}[u](k) = \frac{2^{1-h}|k|^{h-1 / 2} K_{h-\frac{1}{2}}(|k|)}{\Gamma(h)}.
\end{equation}
The $K_v(k)$ stands for the modified Bessel function of the second kind,
satisfying
\begin{equation}\label{eq:1104-0408}
  K_v(k) \propto \sqrt{\frac{\pi}{2}} \frac{e^{-k}}{\sqrt{k}}\left(1+\mathcal{O}\left(\frac{1}{k}\right)\right), \quad \text{as } |k| \rightarrow \infty.
\end{equation}
From \cref{eq:1104-0407} and \cref{eq:1104-0408} we know
\begin{equation}
  \mathcal{F}[u](k) \lesssim e^{-|k|} \cdot |k|^{h-1}.
\end{equation}
By \cref{thm:L2-proj}, for $\beta = a$, $a$ is a constant,  we have
\begin{equation}
  \|u - \widehat{\Pi}_N^\beta u\| \lesssim N^{\frac{1}{4}-h}.
\end{equation}
Let $C$ be a large constant,
$\beta = C h \ln N / \sqrt{N}$ balances the spatial and frequency
error. By \cref{thm:L2-proj} we have
\begin{equation}\label{eq:1105-0411}
  \|u - \widehat{\Pi}_N^\beta u\| \lesssim \left(N / \ln N\right)^{\frac{1}{2}-2h}.
\end{equation}
Next, we solve the model problem \cref{eq:model-problem} by the
Hermite--Galerkin method defined in \cref{eq:Hemrite-Galerkin} with
$\gamma = 1$, true solution $u = 1/(1+x^2)^h$.
Take $u = 1/(1+x^2)$ as an example, with scaling factors
$\beta = 1$ and $\beta = 10/\sqrt{N}$. Notice that the
scaling factor we choose here is slightly different from \cref{eq:1105-0411}.
If $\beta = C/\sqrt{N}$ where $C$ is a large number, then by \cref{thm:L2-proj},
the frequency truncation error is negligible, in the pre-asymptotic range
we have
\begin{equation}
  \|u - \widehat{\Pi}_N^\beta u\| \lesssim N^{\frac{1}{2}-2h}.
\end{equation}
 The $L^2$ error
$\|u - u_N\|$ is presented in \cref{fig:x-2-beta-cmp}.

\begin{figure}[htbp]
  \centering
  \includegraphics[width=0.6\linewidth]{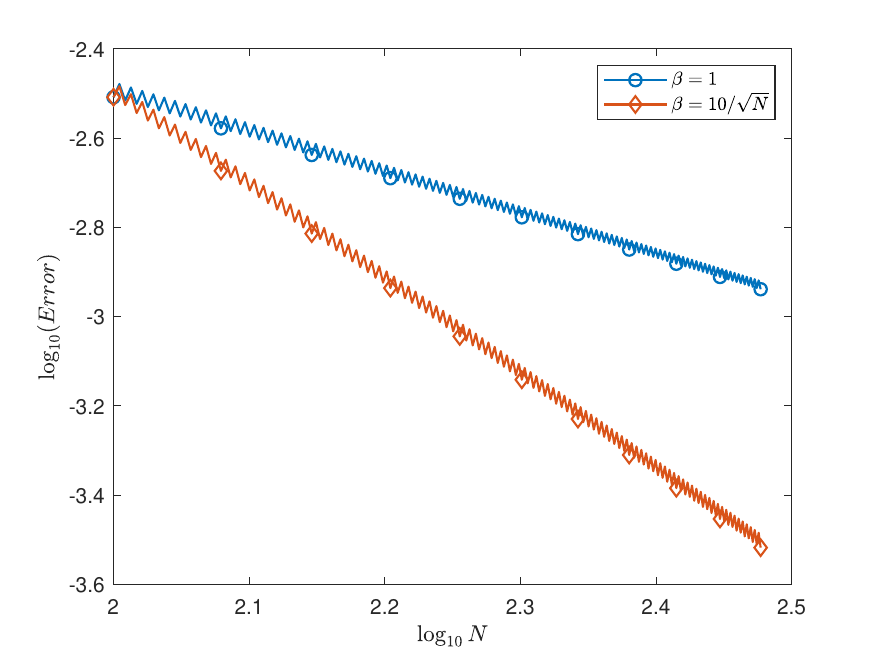}
  \caption{Error $\|u - u_N\|$ in solving model equation \cref{eq:model-problem} with exact solution $u = 1/(1+x^2)$ and 
  scaling factor $\beta=1, \beta = 10/\sqrt{N}$. N ranges from 200 to 300.}
  \label{fig:x-2-beta-cmp}
\end{figure}

For $u = 1/(1+x^2)^h$, taking different $h$, we list the convergence
orders of $\|u - u_N\|$ in \cref{tab:conv-order}. Recall that we use \cref{eq:Hemrite-Galerkin}
to solve the model problem. Here we choose $\beta = 5$ and $\beta = 30/\sqrt{N}$,
with truncated terms that satisfy $200 \leqslant N \leqslant 400$.

\begin{table}[htbp]
\caption{Convergence order of different scaling factor choices for exact solution $u=1/(1+x^2)^h$.}\label{tab:conv-order}
\begin{center}
  \begin{tabular}{|c|c|c|} \hline
   $h$ & $\beta = 5$ & $\beta = 30/\sqrt{N}$ \\ \hline
    1.0 & 0.940 & 2.04 \\
    1.4 & 1.32 & 2.83 \\ 
    1.8 & 1.70 & 3.62 \\ 
    2.2 & 2.07 & 4.41 \\ 
    2.6 & 2.45 & 5.20 \\
    3.0 & 2.83 & 5.99 \\\hline
  \end{tabular}
\end{center}
\end{table}

It can be clearly seen from \cref{tab:conv-order} that proper scaling doubles the convergence order. This fact
holds for all functions that have algebraic(exponential) decay
in the spatial domain and exponential(algebraic) decay in the frequency
domain. 

\subsection{Why error in pre-asymptotic range exhibits sub-geometric convergence}
\label{subsec: puzzling convergence}
When solving the model problem \cref{eq:model-problem} with a true solution
$u = 1/(1+x^2)^h$, \cite{shen_new_2000} reported sub-geometric convergence $\exp(-c\sqrt{N})$ for moderate $N$.
This is puzzling, as the classical error estimate
only predicts a convergence rate of about $N^{-h}$.

In our error analysis framework, the sub-geometric convergence
that occurs in the pre-asymptotic range is
a natural result of \cref{thm:L2-proj}. By \cref{eq:1104-0407} and
\cref{eq:1104-0408}, the frequency truncation error satisfying
\begin{equation}
  \bigl\|\mathcal{F}[u](k) \cdot \mathbb{I}_{\{|k|>b \sqrt{N}\}}\bigr\| \lesssim e^{-c\sqrt{N}},
\end{equation}
while the spatial truncation error satisfies
\begin{equation}
  \bigl\|u \cdot \mathbb{I}_{\{|x|>a \sqrt{N}\}}\bigr\| \lesssim N^{\frac{1}{4}-h}.
\end{equation}
Although asymptotically, the spatial truncation error will be the dominant term, hence the error has an order of about $N^{-h}$,
in the pre-asymptotic range, the frequency truncation error can be
larger than the spatial truncation error, making the total error
show a sub-geometric convergence order.

To make our argument more convincing, we will calculate the position
where the error changes from a sub-geometric convergence $\exp(-c\sqrt{N})$
to an algebraic convergence about $N^{-h}$, then compare our results
with results of numerical experiments.

Recall that $\widehat{I}_N u$ interpolates $u$ at collocation points
$\{x_j\}_{j=0}^N$, which are the roots of $\widehat{H}_{N+1}$. We call the interval $\left[x_0, x_N\right]$ the collocation interval.
Since information outside the collocation interval is not used in interpolation, we cannot
expect a good approximation when the spatial truncation error or frequency truncation error
outside $\left[x_0, x_N\right]$, roughly $\bigl[-\sqrt{2N},\sqrt{2N}\bigr]$, is not negligible.

Based on the above analysis, it is appropriate to choose $\bigl[-\sqrt{2N},\sqrt{2N}\bigr]$ as the
interval to calculate the spatial and frequency truncation error.
The error transforms from sub-geometric convergence to algebraic
convergence when the spatial truncation error equals the frequency truncation error, i.e.,
\begin{equation}
  \bigl\|u \cdot \mathbb{I}_{\{|x|>\sqrt{2N}\}}\bigr\| = \bigl\|\mathcal{F}[u](k) \cdot \mathbb{I}_{\{|k|>\sqrt{2N}\}}\bigr\|.
\end{equation}
For $u_h = 1/(1+x^2)^h$, we find $N$ such that
\begin{equation}\label{eq:1106-0416}
  f_h\left(\sqrt{2N}\right) = \bigl\|u \cdot \mathbb{I}_{\{|x|>\sqrt{2N}\}}\bigr\| - \bigl\|\mathcal{F}[u](k) \cdot \mathbb{I}_{\{|k|>\sqrt{2N}\}}\bigr\| = 0.
\end{equation}
The roots of $f_h$ with different $h$ are listed in \cref{tab:fh-root}.

\begin{table}[htbp]
\caption{Roots of $f_h$ defined in \cref{eq:1106-0416}.}\label{tab:fh-root}
\begin{center}
  \begin{tabular}{|c|c|c|c|c|} \hline
    $h$ & 1.5 & 2.0 & 2.5 & 3.0 \\ 
    \hline
    root & 5.92 & 11.1 & 16.7 & 22.5 \\ \hline   
  \end{tabular}
\end{center}
\end{table}

We solve the model problem \cref{eq:model-problem} using
the Hermite--Galerkin method defined in \cref{eq:Hemrite-Galerkin} with
$\gamma = 1$, true solution $u = 1/(1+x^2)^h$. Taking different
$h$, the behavior of $\|u - u_N\|$ is presented in \cref{fig:x-2h-cmp}.

\begin{figure}[htbp]
  \centering
  \includegraphics[width=0.6\linewidth]{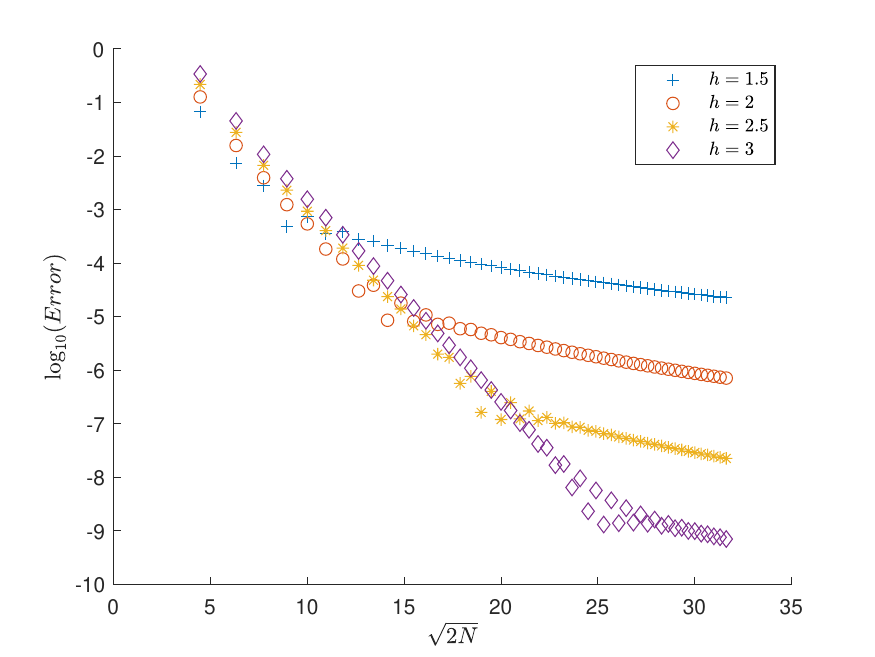}
  \caption{Error $\|u - u_N\|$ in solving model problem 
  \cref{eq:model-problem} with exact solution $u = 1/(1+x^2)^h$ for
  different values of parameter $h$.}
  \label{fig:x-2h-cmp}
\end{figure}

From \cref{tab:fh-root} and \cref{fig:x-2h-cmp} we find that the root of
$f_h$ perfectly matches the position where the error transforms
from sub-geometric to algebraic convergence, this proves the validity
of our theory.

\section{Conclusions}
\label{sec:conclusions}
In this paper, we present a systematic error analysis framework for the scaled Hermite approximation, providing an a priori criterion to select the optimal scaling factor. Taking the $L^2$
projection error as an example, our results
demonstrate that when approximating a function $u$ using the first $N+1$ terms of the scaled Hermite functions $\widehat{H}_n(\beta x)$, the total error consists of three distinct components: spatial truncation error $\|u \cdot \mathbb{I}_{\{|x| \geqslant a\sqrt{N}/\beta\}}\|$,
frequency truncation
error $\|\mathcal{F}[u](k) \cdot \mathbb{I}_{\{|k| \geqslant b \sqrt{N} \beta\}}\|$, and Hermite spectral error
$\|u\|e^{-cN}$. Here, $a,b,c$ are all fixed constants.

Within this framework, determining the optimal scaling factor involves balancing the spatial and frequency truncation error. As a practical application, we demonstrate the optimality of the scaled Gauss--Hermite quadrature, thereby addressing concerns raised in the literature regarding the inefficiency of the Gauss--Hermite quadrature. 
Moreover, we show that proper scaling can restore geometric convergence for functions that decay sufficiently rapidly in both the spatial and frequency domains.
For smooth functions with algebraic decay, our analysis indicates that proper scaling can double the convergence order.
The puzzling pre-asymptotic sub-geometric convergence has been observed more than two decades when approximating algebraically
decaying functions is now clearly explained within our framework.
This framework and the new results can immediately be used to improve the error bounds of applying Hermite methods for important applications (e.g. uncertainty quantification~\cite{xiu_modeling_2003,BabusNT2007StochasticCollocation}) in the existing literature.

It is worth noting that, due to Hardy's uncertainty principle, the proposed framework cannot achieve super-geometric convergence. However, such super-geometric convergence may only be attainable when approximating a very special set of functions, which deserves a special treatment.


\section*{Acknowledgments}
We would like to acknowledge financial support from National Natural Science Foundation of China under Grant No. 12171467, 12494543, 12161141017, Strategic Priority Research Program of Chinese Academy of Sciences under Grant XDA0480504 and National Key 	R\&D Program of China under Grant No. 2021YFA1003601.

\bibliographystyle{siamplain}
\bibliography{Hermite_scaling_paper}
\end{document}